\documentclass[11pt]{article}
\usepackage{amssymb}
\usepackage{amsfonts}
\usepackage{amsmath}
\usepackage{amscd}
\usepackage{amsthm}
\usepackage{latexsym}
\usepackage[all]{xy}
\usepackage{CJK}
\usepackage{setspace}
\usepackage{indentfirst}
\usepackage{mathabx}
\numberwithin{equation}{section}
\newtheorem{theorem}{Theorem}
\newtheorem{lemma}{Lemma}[section]
\newtheorem{thm}[lemma]{Theorem}
\newtheorem{prop}[lemma]{Proposition}
\newtheorem{definition}[lemma]{Definition}
\newtheorem{corollary}[lemma]{Corollary}
\newtheorem{remark}[lemma]{Remark}

\newcommand{\add}{\rm{add}}
\newcommand{\pd}{\rm{pd}}
\newcommand{\Ext}{\rm{Ext}}
\newcommand{\End}{\rm{End}}
\newcommand{\Hom}{\rm{Hom}}
\newcommand{\mo}{\rm{mod}}
\newcommand{\id}{\rm{id}}
\newcommand{\dom}{\rm{domdim}}
\newcommand{\gl}{\rm{gld}}
\newcommand{\op}{\rm{op}}

\begin{document}
\title{\bf Algebras with finite relative dominant dimension and almost n-precluster tilting modules }
\author{Shen Li\quad and\quad  Shunhua Zhang}
\date{}
\maketitle

\begin{center}\section*{}\end{center}

\begin{abstract}
In this paper, we investigate the relative dominant dimension with respect to an injective module and characterize the algebras with finite relative dominant dimension. As an application, we introduce the almost $n$-precluster tilting module and establish a correspondence between  almost $n$-precluster tilting modules and almost $n$-minimal Auslander-Gorenstein algebras. Moreover, we give a description of the Gorenstein projective modules over almost $n$-minimal Auslander-Gorenstein algebras in terms of the corresponding almost $n$-precluster tilting modules.
\end{abstract}

{\small {\bf Key words and phrases:} \ {\rm Relative dominant dimension, Almost $n$-minimal Auslander-Goresntein algebra, Almost $n$-precluster tilting module}}

{\small {\bf Mathematics Subject Classification 2010:} \ {\rm 16E10, 16E65, 16G10}}

\vskip0.2in

\section{Introduction}

Let $\Lambda$ be an Artin algebra and $I$ be an injective $\Lambda$-module. We say the relative dominant dimension of $\Lambda$ with respect to $I$ is at least $n+1$ if $\Lambda$ has a minimal injective resolution $0 \rightarrow \Lambda \rightarrow I_{0} \rightarrow I_{1} \rightarrow \cdots \rightarrow I_{n}\rightarrow \cdots$ with $I_{0},I_{1},\ldots, I_{n} \in$ \add\,$I$, denoted by $I$-${\dom}\,\Lambda\geq n+1$. If $I$ is a maximal projective-injective summand of $\Lambda$, $I$-${\dom}\,\Lambda$ is called the dominant dimension of $\Lambda$. The dominant dimension of an algebra, as a kind of  important homological dimension, measures how far this algebra is from being self-injective. In representation theory, algebras with finite  dominant dimension are of particular interest. For instance, algebras with finite dominant were characterized by the existence of certain tilting modules in \cite{NRTZ} and \cite{PS}. Any algebra with dominant dimension at least 2 is the endomorphism algebra of a generator-cogenerator over an algebra, see \cite{AS} and \cite{BJM}. Iyama and Solberg studied algebras whose dominant dimensions are equal to their Gorenstein dimensions in \cite{IS}. Then it is natural to consider whether we can generalize these results to algebras with finite relative dominant dimension. In this paper, we investigate the relative dominant dimension with respect to an injective module whose projective dimension is finite. The main tool that we use is the relative homology and relative cotilting theory introduced in \cite{AS1,AS2, AS3}.

Assume $I$ has projective dimension at most $m$, then $I$-${\dom}\,\Lambda\geq n+1$ implies that $\Lambda$ satisfies the $(m+1,n+1)$-condition which is introduced in \cite{I1}. Let $0 \rightarrow \Lambda \rightarrow I_{0} \rightarrow I_{1} \rightarrow \cdots \rightarrow I_{n}\rightarrow \cdots$ be a minimal injective resolution of $\Lambda$. Recall that an algebra $\Lambda$ is said to satisfy the $(m+1,n+1)$-condition if ${\pd}\,I_{j}\leq m$ holds for $0\leq j \leq n$.  Since the $(m+1,n+1)$-condition is not left-right symmetric, we say $\Lambda$ satisfies the two-sided $(m+1,n+1)$-condition if both $\Lambda$ and $\Lambda^{\op}$ satisfy the $(m+1,n+1)$-condition. Iyama established the relative version of higher Auslander correspondence to characterize algebras satisfying the two-sided $(m+1,n+1)$-condition, see \cite{I2} for details. We study and give a characterization of algebras satisfying the one-sided $(m+1,n+1)$-condition, that is algebras $\Lambda$ with $I$-${\dom}\,\Lambda\geq n+1$ for an injective module $I$ whose projective dimension is at most $m$.

First we show how to construct the algebra $\Lambda$ with $I$-${\dom}\,\Lambda\geq n+1$.

\begin{theorem}{\rm (Theorem 3.3)}Let $A$ be an Artin algebra, $M$ be an $A$-module and $\Lambda={\End}_{A}\,M$. If $M$ satisfies $(i)$ $DA\in {\add}\,M$, $(ii)$ ${\Ext}^{i}_{A}(M,M)=0$ for $1\leqslant i \leqslant n-1$ $(n\geq 1)$ and $(iii)$ $M$-{\rm codim}$\,A\leqslant m$ $(m\geqslant 0)$, we have the following.

$(1)$ $_{\Lambda}I={\Hom}_{A}(A,M)$ is an injective $\Lambda$-module with ${\pd}_{\Lambda}\,I\leqslant m$.

$(2)$ $A\cong {\End}_{\Lambda}I$ and $I\text{-}{\dom}\,\Lambda\geq n+1$.
\end{theorem}

Conversely, we prove that any algebra $\Lambda$ with $I\text{-}{\dom}\,\Lambda\geq n+1$ is the endomorphism algebra of a cogenerator satisfying certain conditions.

\begin{theorem}{\rm (Theorem 3.6)}Let $\Lambda$ be an Artin algebra, $I$ be an injective $\Lambda$-module with projective dimension at most $m$ $(m\geq0)$ and $A={\End}_{\Lambda}I$. If $I$-{\rm domdim}\,$\Lambda\geq n+1$ $(n\geq 1)$, we have the following.

$(1)$ $_{A}M={\Hom}_{\Lambda}(\Lambda, I)$ is a cogenerator over $A$ and $M$-{\rm codim}$\,A\leqslant m$.

$(2)$ ${\Ext}^{i}_{A}(M,M)=0$ for $1\leq i\leq n-1$.

$(3)$ $\Lambda\cong{\End}_{A}\,M$.
\end{theorem}

The above results provide a characterization of algebras with finite relative dominant dimension.

Now we focus our attention on a special kind of algebra with finite relative dominant dimension. Let $I$ be the direct sum of all pairwise non-isomorphic indecomposable injective $\Lambda$-modules with projective dimension at most 1. We say $\Lambda$ is an almost $n$-minimal Auslander-Gorenstein algebra (respectively, almost $n$-Auslander algebra) if it satisfies ${\id}\,\Lambda\leq n+1 \leq I\text{-}{\dom}\,\Lambda$ (respectively, ${\gl}\,\Lambda\leq n+1 \leq I\text{-}{\dom}\,\Lambda$ ). This kind of algebra was introduced by Adachi and Tsukamoto in \cite{AT} and was characterized by the existence of certain tilting-cotilting modules. Based on Theorem 1 and Theorem 2, we want to characterize the almost $n$-minimal Auslander-Gorenstein algebra from the viewpoint of higher Auslander correspondence (\cite{I2},\cite{IS}). This is the starting point of our paper.

The almost $n$-minimal Auslander-Gorenstein algebra is a generalization of the $n$-minimal Auslander-Gorenstein algebra. In \cite{IS}, Iyama and Solberg established a correspondence between $n$-precluster tilting modules and $n$-minimal Auslander-Gorenstein algebras. We generalize the $n$-precluster tilting module to the almost $n$-precluster tilting module (Definition 4.2) and prove the following correspondence for almost $n$-minimal Auslander-Gorenstein algebras. This is the main result of our paper.

\begin{theorem}{\rm (Theorem 4.11)}Assume $n\geq2$. There exists a bijection between the equivalence classes of almost n-precluster tilting modules over Artin algebras and the Morita-equivalence classes of almost n-minimal Auslander-Gorenstein algebras.
\end{theorem}

Moreover, we show some deep relations between almost $n$-precluster tilting modules and almost $n$-minimal Auslander-Gorenstein algebras. For an almost $n$-minimal Auslander-Gorenstein algebra $\Lambda$, we describe the Gorenstein projective $\Lambda$-modules in terms of the corresponding $A$-module $_{A}M={\Hom}_{\Lambda}(\Lambda,I)$ where $A={\End}_{\Lambda}I$.

\begin{theorem}{\rm (Theorem 4.13)}Let $F=F^{M}$ be an additive sub-bifunctor of ${\Ext}^{1}_{A}(-,-)$. Then the functor ${\Hom}_{A}(-,M)$ induces a duality between $M^{\perp_{F}}$ and $\mathcal{GP}(\Lambda)$. Moreover, $M^{\perp_{F}}/[M]$ and $\underline{\mathcal{GP}(\Lambda)}^{\op}$ are equivalent as triangulated categories.
\end{theorem}

Finally, we investigate a special class of almost $n$-precluster tilting modules (Definition 4.16) which correspond to almost $n$-Auslander algebras. For the case $n=2$, all the modules in this class share the same number of pairwise non-isomorphic indecomposable direct summands (Proposition 4.17).

This paper is organized as follows. In section 2, we recall the relative homology and relative cotilting theory which are the main tool in our paper. Section 3 is devoted to the proof of Theorem 1 and Theorem 2. In section 4, we introduce the almost $n$-precluster tilting module. Then we prove Theorem 3 and Theorem 4. The last part of section 4 gives some results on almost $n$-cluster tilting modules.

\section{Preliminaries}
Throughout this paper, all algebras are Artin algebras and all modules are finitely generated left modules. $D$ is the Matlis dual. For an algebra $A$, $A$-mod is the category of finitely generated $A$-modules. For an $A$-module $M$, add\,$M$ is the full subcategory of $A$-mod consisting of direct summands of finite direct sums of $M$. $M$ is called a generator (respectively, cogenerator) if $A\in {\add}M$ $({\rm respectively}, DA\in {\add}M)$. We denote by pd\,$M$ (respectively, id\,$M$) the projective dimension (respectively, injective dimension) of $M$. The composition of morphisms $f:X \rightarrow Y$ and $g:Y \rightarrow Z$ is denoted by $gf:X \rightarrow Z$.

The aim of this section is to recall some basic definitions and results on relative homology and relative cotilting theory. For a systematic study, refer to \cite{AS1,AS2,AS3}.

\subsection{Relative homology}
Let $A$ be an algebra and $F$ be an additive sub-bifunctor of ${\Ext}^{1}_{A}(-,-):(A\text{-}{\mo})^{\op}\times A\text{-}{\mo} \rightarrow \mathcal{A}b$, where $\mathcal{A}b$ is the category of abelian groups. A short exact sequence $0\rightarrow X \rightarrow Y \rightarrow Z \rightarrow 0$ is called $F$-exact if it is in $F(Z,X)$. An $A$-module $P$ is called $F$-projective if all $F$-exact sequences $0 \rightarrow X \rightarrow Y \rightarrow P \rightarrow 0$ split. The $F$-injective module is defined dually. Denote by $\mathcal{P}(F)$ $({
\rm respectively}, \mathcal{P}(A))$ the subcategory of $F$-projective (respectively, projective) modules in $A$-mod and by $\mathcal{I}(F)$ $({\rm respetively}, \mathcal{I}(A))$ the subcategory of $F$-injective (respectively, injective) modules in $A$-mod, we have $\mathcal{P}(F)=\mathcal{P}(A)\cup\tau^{-}\mathcal{I}(F)$ and $\mathcal{I}(F)=\mathcal{I}(A)\cup \tau\mathcal{P}(F)$ by \cite[Corollary 1.6]{AS1}. $F$ is said to have enough $F$-projectives (respectively, $F$-injectives) if for any $X\in A$-mod there exists an $F$-exact sequence $0 \rightarrow Y \rightarrow P \rightarrow X \rightarrow 0$ with $P\in \mathcal{P}(F)$ $({\rm respectively}, 0\rightarrow X\rightarrow I \rightarrow Y \rightarrow 0\ with \ I \in \mathcal{I}(F))$.

Let $M$ be an $A$-module. In this paper, we work with additive sub-bifunctors of ${\Ext}^{1}_{A}(-,-)$ of the form $F^{M}$ and $F_{M}$. For each pair of $A$-modules $X$ and $Z$, we define
\begin{multline*}
F^{M}(Z,X)=\{ 0\rightarrow X \rightarrow Y \rightarrow Z \rightarrow 0 \in {\Ext}^{1}_{A}(Z,X) \\
\mid {\Hom}_{A}(Y,N)\rightarrow {\Hom}_{A}(X,N)\rightarrow 0 \text{ is exact for all}\ N \in {\add}M \}.
\end{multline*}
$F_{M}$ is defined dually. It is shown in \cite{AS1} that $F^{M}$ and $F_{M}$ are additive sub-bifunctors of ${\Ext}^{1}_{A}(-,-)$ and they have enough $F$-projectives and $F$-injectives. Moreover, $\mathcal{P}(F_{M})={\add}(A\oplus M)$ and $\mathcal{I}(F^{M})={\add}(DA\oplus M)$.

For an $A$-module $X$, an exact sequence $\cdots \rightarrow P_{l} \xrightarrow{f_{l}} P_{l-1}\rightarrow \cdots \rightarrow P_{1} \xrightarrow{f_{1}} P_{0}\xrightarrow{f_{0}} X\rightarrow 0$ is called an $F$-projective resolution of $X$ if all $P_{i}\in \mathcal{P}(F)$ and all $0 \rightarrow \text{Im}f_{i+1} \rightarrow P_{i} \rightarrow \text{Im}f_{i} \rightarrow 0$ are $F$-exact. The $F$-projective dimension of $X$, denoted by pd$_{F}X$, is the smallest $n$ such that there exists an $F$-projective resolution $0 \rightarrow P_{n} \rightarrow \cdots \rightarrow P_{1} \rightarrow P_{0} \rightarrow X \rightarrow 0$. Dually, we can define the $F$-injective resolution and $F$-injective dimension of $X$. Similar to the extension group ${\Ext}^{i}_{A}(X,Y)$, by using the $F$-projective resolution of $X$ or the $F$-injective resolution of $Y$, we can define the $F$-relative extension group ${\Ext}^{i}_{F}(X,Y)$ and they share many basic properties. Then the $F$-global dimension of $A$ is defined as ${\gl}_{F}A=\text{sup}\{{\pd}_{F}X\mid X\in A\text{-mod}\}=\text{sup}\{{\id}_{F}Y|Y\in A\text{-mod}\}$.

For a subcategory $\mathcal{C}$ of $A$-mod, we write $^{\perp_{F}}\mathcal{C}=\{X\in A\text{-}{\mo}\mid{\Ext}^{i}_{F}(X,\mathcal{C})\\ =0 \,\text{for all} \ i>0\}$ and $\mathcal{C}^{\perp_{F}}=\{X \in A\text{-mod}\mid {\Ext}^{i}_{F}(\mathcal{C},X)=0 \ \text{for all} \ i>0\}$. If $F={\Ext}^{1}_{A}(-,-)$, we  denote $^{\perp_{F}}\mathcal{C}$ and $\mathcal{C}^{\perp_{F}}$ by $^{\perp}\mathcal{C}$ and $\mathcal{C}^{\perp}$ respectively. For an integer $n$, we write $^{\perp_{n}}\mathcal{C}=\{X \in A\text{-mod}\mid {\Ext}^{i}_{A}(X,\mathcal{C})=0 \ \text{for} \ 1\leq  i\leq n\}$ and $\mathcal{C}^{\perp_{n}}=\{X \in A\text{-mod}\mid {\Ext}^{i}_{A}(\mathcal{C},X)=0 \ \text{for } \ 1\leq i\leq n\}$.

\subsection{Relative cotilting theory}

Let $F$ be an additive sub-bifunctor of ${\Ext}^{1}_{A}(-,-)$. We recall the definition of $F$-cotilting modules.

\begin{definition}An $A$-module $T$ is called $F$-cotilting if it satisfies

$(1)$ ${\id}_{F}T<\infty$,

$(2)$ ${\Ext}^{i}_{F}(T,T)=0$ for $i>0$,

$(3)$ There exists an $F$-exact sequence $0\rightarrow T_{l}\rightarrow \cdots \rightarrow T_{1} \rightarrow T_{0} \rightarrow I \rightarrow 0$ with $T_{i}\in {\add}\,T$ for any $I$ in $\mathcal{I}(F)$.
\end{definition}

We need the following results on $F$-cotilting modules later.

\begin{thm}{\rm (\cite{AS2})} Let $T$ be an $F$-cotilting $A$-module and $\Lambda={\End}_{A}T$, then we have the following.

$(1)$ $A\cong{\End}_{\Lambda}T$.

$(2)$ The module $U={\Hom}_{A}(\mathcal{P}(F),T)$ is a cotilting $\Lambda$-module with ${\id}_{F}T\leq {\id}_{\Lambda}U\leq{\id}_{F}T+2$.

$(3)$ ${\Hom}_{A}(-,T):A\text{-}{\rm mod}\rightarrow \Lambda\text{-}{\rm mod}$ induces a duality between $^{\perp_{F}}T$ and $^{\perp}U$ and ${\Ext}^{i}_{F}(X,Y)\cong {\Ext}^{i}_{\Lambda}({\Hom}_{A}(Y,T),{\Hom}_{A}(X,T))$ for all modules $X$ and $Y$ in $^{\perp_{F}}T$ and $i\geq 0$.
\end{thm}

\begin{proof}
(1) is \cite[Corollary 3.4]{AS2} and (2) is \cite[Theorem 3.13(d)]{AS2}. (3) follows from \cite[Corollary 3.6, Proposition 3.7]{AS2}.
\end{proof}

Let $X$ be an $A$-module and $\Gamma={\End}_{A}X$. For an $A$-module $Y$, we say $X$ dualizes $Y$ if the natural homomorphism $Y \rightarrow {\Hom}_{\Gamma}({\Hom}_{A}(Y,X),X)$ is an isomorphism. This is equivalent to that there exists an exact sequence $0 \rightarrow Y \xrightarrow{f} X^{0} \rightarrow X^{1}$ where $f$ is a left add$X$-approximation of $Y$ and $X^{0},X^{1} \in {\add}X$ by \cite[Proposition 2.1]{AS3}. If $X$ is a direct summand of $Y$, we call $X$ is a dualizing summand of $Y$.

Given a dualizing summand of a cotilting module, we can get a relative cotilting module.
\begin{thm}{\rm \cite[Theorem 2.10]{IS}}Let $\Lambda$ be an Artin algebra and $U$ be a cotilting $\Lambda$-module. For a dualizing summand $T$ of $U$, let $A={\End}_{\Lambda}T$, $X={\Hom}_{\Lambda}(U,T)$ and $F=F_{X}$ be an additive sub-bifunctor of ${\Ext}^{1}_{A}(-,-)$. Then the following hold.

$(1)$ $\Lambda\cong{\End}_{A}T$.

$(2)$ $T$ is an $F$-cotilting $A$-module with ${\id}_{F}T\leq {\rm max}\{{\id}_{\Lambda}U, \ 2\}$.

$(3)$ If $T$ is an injective $\Lambda$-module, then ${\id}_{F}T\leq {\rm max}\{{\id}_{\Lambda}U-2, \ 0\}$.
\end{thm}

$F$-Gorenstein algebras were introduced by Auslander and Solberg in \cite{AS} to characterize Gorenstein endomorphism algebras of $F$-cotilting modules. Let $\mathcal{P}^{\infty}(F)=\{X\in A\text{-}{\mo}|\,{\pd}_{F}X<\infty\}$ and $\mathcal{I}^{\infty}(F)=\{Y\in A\text{-}{\mo}| \,{\id}_{F}Y<\infty\}$. Recall that an Artin algebra $A$ is called $F$-Gorenstein if $\mathcal{P}^{\infty}(F)=\mathcal{I}^{\infty}(F)$.

\begin{prop}{\rm \cite[Proposition 3.6]{AS}}Let $T$ be an $F$-cotilting $A$-module and $\Lambda={\End}_{A}T$. Then the following are equivalent.

$(1)$ $\Lambda$ is Gorenstein.

$(2)$ $T$ is an $F$-tilting $A$-module.

$(3)$ $A$ is $F$-Gorenstein.
\end{prop}

\section{Algebras with finite relative dominant dimension}

Let $\Lambda$ be an Artin algebra and $I$ be an injective $\Lambda$-module whose projective dimension is at most $m$. In this section, we characterize the algebra $\Lambda$ with $I$-domdim\,$\Lambda\geq n+1$.

First we give a description of the injective $\Lambda$-modules whose projective dimensions are at most $m$.

\begin{prop} Let $0 \rightarrow \Lambda \rightarrow I_{0} \rightarrow I_{1} \rightarrow \cdots \rightarrow I_{n} \rightarrow I_{n+1} \rightarrow \cdots$ be a minimal injective resolution of $\Lambda$ and $I$-{\dom}\,$\Lambda\geq n+1$. If $m\leq n$, then ${\add}\,I$ consists of all the injective $\Lambda$-modules whose projective dimensions are at most $m$.
\end{prop}

\begin{proof}
Since $I$-domdim\,$\Lambda\geq n+1$, we have $\oplus^{n}_{j=0}I_{j}\in {\add}\,I$. For an injective $\Lambda$-module $J$ with ${\pd}\,J\leq m$, $J\in {\add}(\oplus^{m}_{j=0}I_{j})$ holds according to \cite[Theorem 2.2]{JM}. Note that $m\leq n$ and ${\pd}\,I\leqslant m$, we know that ${\add}\,I$ consists of all the injective $\Lambda$-modules whose projective dimensions are at most $m$.
\end{proof}

\begin{remark} \rm When $m>n$, {\add}\,$I$ does not always contain all the injective modules with projective dimension at most $m$. Let $\Gamma$ be the algebra given by the quiver $1 \xrightarrow{\alpha} 2 \xrightarrow{\beta} 3 \xrightarrow{\gamma}4\leftarrow 5\leftarrow 6$ and the relation $\gamma\beta\alpha=0$. Take $I=\oplus^{6}_{j=2}I(j)$. It is easy to check that ${\pd}\,I=2$ and $I$-domdim\,$\Gamma=2$. However, ${\pd}\,I(1)=2$ and $I(1)\notin {\add}\,I$.
\end{remark}

Let $A$ be an Artin algebra and $M$ be a cogenerator over $A$. For an $A$-module $X$, an exact sequence $0 \rightarrow X \xrightarrow{f_{0}} M_{0}\xrightarrow{f_{1}} M_{1} \xrightarrow{f_{2}} M_{2}\xrightarrow{f_{3}}\cdots $ with $M_{i}\in {\add}\,M$ is called an $M$-coresolution of $X$. Moreover, an $M$-coresolution of $X$ is called minimal if $f_{i}$ is a minimal left ${\add}\,M$-approximation for all $i\geqslant 0$. The $M$-coresolution of $X$ always exists since $M$ is a cogenerator. The $M$-coresolution dimension of $X$, denoted by $M$-codim\,$X$, is the number $l$ such that there exists a minimal $M$-coresolution $0 \rightarrow X \rightarrow M_{0}\rightarrow \cdots \rightarrow M_{l-1}\rightarrow M_{l}\rightarrow 0$.

In \cite{AS}, Auslander and Solberg proved that the endomorphism algebra of a generator-cogenerator has dominant dimension at least 2. Now we show how to construct algebras with finite relative dominant dimension.

\begin{thm}
Let $A$ be an Artin algebra, $M$ be an $A$-module and $\Lambda={\End}_{A}\,M$. If $M$ satisfies $(i)$ $DA\in {\add}\,M$, $(ii)$ ${\Ext}^{i}_{A}(M,M)=0$ for $1\leqslant i \leqslant n-1$ $(n\geq 1)$ and $(iii)$ $M$-{\rm codim}$\,A\leqslant m$ $(m\geqslant 0)$, we have the following.

$(1)$ $_{\Lambda}I={\Hom}_{A}(A,M)$ is an injective $\Lambda$-module with ${\pd}_{\Lambda}\,I\leqslant m$.

$(2)$ $A\cong {\End}_{\Lambda}I$ and $I\text{-}{\dom}\,\Lambda\geq n+1$.
\end{thm}

\begin{proof}
Since $_{\Lambda}I={\Hom}_{A}(A,M)\cong D{\Hom}_{A}(M,DA)$ and $DA\in {\add}\,M$, we know that $I$ is an injective $\Lambda$-module. Let $0\rightarrow A \rightarrow M_{0} \rightarrow M_{1} \rightarrow \cdots \rightarrow M_{m}\rightarrow 0$ be a minimal $M$-coresolution of $A$. Applying ${\Hom}_{A}(-,M)$, we get an exact sequence $0 \rightarrow {\Hom}_{A}(M_{m},M)\rightarrow \cdots \rightarrow {\Hom}_{A}(M_{1},M)\rightarrow {\Hom}_{A}(M_{0},M)\rightarrow I \rightarrow 0$ which is a projective resolution of $I$. This implies ${\pd}_{\Lambda}\,I\leqslant m$.

Let $F=F^{M}$ be an additive sub-bifunctor of  ${\Ext}^{1}_{A}(-,-)$. Note that $\mathcal{I}(F^{M})={\add}(DA\oplus M)={\add}\,M$, we know $M$ is an $F$-cotilting module. Then $A\cong {\End}_{\Lambda}I$ follows from Theorem 2.2 (1). Take a minimal projective resolution $\cdots \rightarrow P_{n}\rightarrow \cdots \rightarrow P_{1}\rightarrow P_{0}\rightarrow M \rightarrow 0$ of $M$. Since ${\Ext}^{i}_{A}(M,M)=0$ for $1\leqslant i \leqslant n-1$, applying ${\Hom}_{A}(-,M)$, we get an exact sequence $0\rightarrow \Lambda \rightarrow {\Hom}_{A}(P_{0},M)\rightarrow {\Hom}_{A}(P_{1},M)\rightarrow \cdots \rightarrow {\Hom}_{A}(P_{n},M)$ with ${\Hom}_{A}(P_{i},M)\in {\add\,\Hom}_{A}(A,M)$ $={\add}\,I$ for $0\leq i \leq n$. This implies $I\text{-}{\dom}\,\Lambda\geq n+1$.
\end{proof}

Taking $n=1$ and $m=0$, we recover the following result in \cite{AS} immediately.

\begin{corollary}{\rm \cite[Proposition 2.5]{AS}}Let $T$ be a generator-cogenerator over $A$ and $\Gamma={\End}_{A}T$. Then $\Gamma$ has dominant dimension at least $2$ and ${\add}\,_{\Gamma}T$ consists of all projective-injective $\Gamma$-modules.
\end{corollary}

\begin{proof}
$A \in {\add}\,T$ implies $T$-{\rm codim}$\,A=0$. By Theorem 3.3, we know that $_{\Gamma}I={\Hom}_{A}(A,T)$ is projective and $\Gamma$ has dominant dimension at least $2$. Moreover, according to Proposition 3.1, ${\add}\,I={\add}\,_{\Gamma}T$ consists of all projective-injective modules.
\end{proof}

For an $A$-module $X$, we show how to calculate $I\text{-}{\dom}\,{\Hom}_{A}(X,M)$ in terms of ${\Ext}^{i}_{A}(X,M)$. This is inspired by \cite[Lemma 3]{BJM} and \cite[Theorem 1.2]{M}.

\begin{prop}
Keep the notations in Theorem {\rm 3.3} and let $I\text{-}{\dom}_{\geq 2}(\Lambda)$ $=\{Y\in \Lambda\text{-}{\mo}\mid I\text{-}{\dom}\,Y\geq2\}$. The functor ${\Hom}_{A}(-,M)$ induces a duality between $A$-{\mo} and $I\text{-}{\dom}_{\geq 2}(\Lambda)$. Moreover, $I\text{-}{\dom}\,{\Hom}_{A}(X,M)$ $={\rm inf}\{i\geq1\mid {\Ext}^{i}_{A}(X,M)\neq0\}+1$ for an $A$-module $X$.
\end{prop}

\begin{proof}
For an $A$-module $X$, let $l={\rm inf}\{i\geq1\mid {\Ext}^{i}_{A}(X,M)\neq0\}$ and $\cdots \rightarrow P_{l} \rightarrow P_{l-1}\rightarrow \cdots \rightarrow P_{0}\rightarrow X \rightarrow 0$ be a minimal projective resolution of $X$. Applying ${\Hom}_{A}(-,M)$ yields an injective resolution $0 \rightarrow {\Hom}_{A}(X,M)\rightarrow {\Hom}_{A}(P_{0},M)\rightarrow \cdots \rightarrow {\Hom}_{A}(P_{l-1},M)\rightarrow {\Hom}_{A}(P_{l},M) $ with ${\Hom}_{A}(P_{i},M)\in{\add}\,I$ for $0\leq i\leq l$. This implies $I\text{-}{\dom}\,{\Hom}_{A}(X,M)$ $\geq l+1$. If $I\text{-}{\dom}\,{\Hom}_{A}(X,M)>l+1$, we get ${\rm inf}\{i\geq1\mid {\Ext}^{i}_{A}(X,M)\neq0\}>l$, a contradiction. Thus $I\text{-}{\dom}\,{\Hom}_{A}(X,M)=l+1\geq 2$.

Note that $M$ is a cogenerator, according to the dual of \cite[Lemma 1.3(b)]{APR}, ${\Hom}_{A}(-,M)$ is fully faithful. We only need to show it is dense. For a $\Lambda$-module $Y$ with $I\text{-}{\dom}\,Y\geq2$, there exists an exact sequence $0\rightarrow Y \rightarrow {\Hom}_{A}(Q_{0},M)\xrightarrow{f^{*}}{\Hom}_{A}(Q_{1},M)$ where $Q_{0}$ and $Q_{1}$ are projective $A$-modules. By Yoneda's lemma, there is an $f\in{\Hom}_{A}(Q_{1},Q_{0})$ such that $f^{*}={\Hom}_{A}(f,M)$. Applying ${\Hom}_{A}(-,M)$ to $Q_{1} \xrightarrow{f} Q_{0}\rightarrow {\rm coker}f\rightarrow 0$, we get an exact sequence $0\rightarrow {\Hom}_{A}({\rm coker}f,M)\rightarrow {\Hom}_{A}(Q_{0},M)\xrightarrow{f^{*}}
{\Hom}_{A}(Q_{1},M)$. Thus $Y\cong{\Hom}_{A}({\rm coker}f,M)$ and this completes our proof.
\end{proof}

Conversely, any algebra with finite relative dominant dimension is  the endomorphism algebra of a cogenerator satisfying certain conditions.

\begin{thm}Let $\Lambda$ be an Artin algebra, $I$ be an injective $\Lambda$-module with projective dimension at most $m$ $(m\geq0)$ and $A={\End}_{\Lambda}I$. If $I$-{\rm domdim}\,$\Lambda\geq n+1$ $(n\geq 1)$, we have the following.

$(1)$ $_{A}M={\Hom}_{\Lambda}(\Lambda, I)$ is a cogenerator over $A$ and $M$-{\rm codim}$\,A\leqslant m$.

$(2)$ ${\Ext}^{i}_{A}(M,M)=0$ for $1\leq i\leq n-1$.

$(3)$ $\Lambda\cong{\End}_{A}\,M$.
\end{thm}

\begin{proof}
 Since $I$-{\rm domdim}\,$\Lambda\geq n+1$, there exists a minimal injective resolution $0 \rightarrow \Lambda \rightarrow I_{0} \rightarrow I_{1} \rightarrow \cdots \rightarrow I_{n-1} \rightarrow I_{n}\rightarrow \cdots$ with $I_{j}\in{\add}\,I$ for $0\leq j \leq n$. $I_{0},I_{1} \in{\add}\,I$ implies that $I$ dualizes $\Lambda$. Then $\Lambda\cong{\End}_{A}\,M$ holds according to \cite[Lemma 2.1]{AS}. Note that $DA=D{\Hom}_{\Lambda}(I,I)\cong{\Hom}_{\Lambda}(\upsilon^{-}_{\Lambda}I,I)\in {\add}\,{\Hom}_{\Lambda}(\Lambda,I)={\add}\,M$ where $\upsilon_{\Lambda}$ is the Nakayama functor. Thus $M$ is a cogenerator over $A$. Let $0 \rightarrow Q_{m}\xrightarrow{f_{m}} \cdots \rightarrow Q_{1} \xrightarrow{f_{1}} Q_{0}\xrightarrow{f_{0}} I \rightarrow 0$ be a minimal projective resolution of $I$. Since $I$ dualizes $\Lambda$, we know $I$ dualizes the projective $\Lambda$-modules $Q_{i}$ and then $Q_{i}\cong{\Hom}_{A}({\Hom}_{\Lambda}(Q_{i},I),M)$ holds for $0\leq i \leq m$. Thus applying ${\Hom}_{\Lambda}(-,I)$ yields an exact sequence $0 \rightarrow A \xrightarrow{f^{*}_{0}} {\Hom}_{\Lambda}(Q_{0},I)\xrightarrow{f^{*}_{1}} {\Hom}_{\Lambda}(Q_{1},I)\rightarrow \cdots \xrightarrow{f^{*}_{m}} {\Hom}_{\Lambda}(Q_{m},I)\rightarrow 0 $ with ${\Hom}_{\Lambda}(Q_{i},I)\in {\add}\,M$ and $f^{*}_{i}$ left ${\add}\,M$-approximation. It follows that $M$-${\rm codim}\,A\leqslant m$.

 Now we prove ${\Ext}^{i}_{A}(M,M)=0$ for $1\leq i\leq n-1$. Applying ${\Hom}_{\Lambda}(-,I)$ to the minimal injective resolution of $\Lambda$ gives rise to an exact sequence $\cdots \rightarrow {\Hom}_{\Lambda}(I_{n},I)\rightarrow {\Hom}_{\Lambda}(I_{n-1},I)\rightarrow \cdots \rightarrow {\Hom}_{\Lambda}(I_{0},I)\rightarrow M \rightarrow 0$ which is a projective resolution of $M$. In order to calculate ${\Ext}^{i}_{A}(M,M)$, we need to consider the upper row of the following diagram. According to the dual of \cite[Proposition 2.1, \uppercase \expandafter{\romannumeral2} ]{ARS}, we have ${\Hom}_{A}({\Hom}_{\Lambda}(I_{j},I),M)={\Hom}_{A}({\Hom}_{\Lambda}(I_{j},I),{\Hom}_{\Lambda}(\Lambda,I))\cong {\Hom}_{\Lambda}(\Lambda,I_{j})\cong I_{j}$ for $0\leq j \leq n$. Then the following diagram is commutative.

\begin{small}
\[ \xymatrix{
0\ar[r] &(M,M)\ar[d]^\cong \ar[r] &((I_{0},I),M) \ar[d]^\cong \ar[r] & \cdots \ar[r]&((I_{n-1},I),M)\ar[d]^\cong \ar[r]& _{A}((I_{n},I),M)\ar[d]^\cong \\
0\ar[r] &\Lambda  \ar[r] & I_{0} \ar[r] & \cdots \ar[r]& I_{n-1}\ar[r] & I_{n} & }
\]
\end{small}
Since the lower row is exact, the upper row is also exact and then we get ${\Ext}^{i}_{A}(M,M)=0$ for $1\leq i\leq n-1$.
\end{proof}

The above theorem is a generalization of the following result in \cite{AS}.

\begin{corollary}{\rm \cite[Proposition 2.6]{AS}} Let $\Gamma$ be an Artin algebra with dominant dimension at least {\rm 2}, $T$ be the maximal injective direct summand of $\Gamma$ and $A={\End}_{\Gamma}\,T$. Then $_{A}M={\Hom}_{\Gamma}(\Gamma,T)$ is a generator-cogenerator over $A$ and $\Gamma\cong{\End}_{A}\,M$.
\end{corollary}

\begin{proof}
Note that $T\text{-}{\dom}\,\Gamma={\dom}\,\Gamma\geq 2$. By Theorem 3.6, we have $M$-{\rm codim}$\,A=0$ and then $A\in{\add}\,M$. Thus $M$ is a generator-cogenerator and $\Gamma\cong{\End}_{A}\,M$.
\end{proof}

For an Artin algebra $\Lambda$ with $I$-domdim\,$\Lambda\geq n+1$, although we have ${\dom}\,_{\Lambda}\Lambda={\dom}\,\Lambda_{\Lambda}$, there does not necessarily exist an injective right $\Lambda$-module $J$ with projective dimension at most $m$ such that $J_{\Lambda}\text{-}{\dom}\,\Lambda_{\Lambda}\geq n+1$. If such an injective  right $\Lambda$-module exists, $\Lambda$ satisfies the two-sided $(m+1,n+1)$-condition which is introduced by Iyama in \cite{I1}. He established the $m$-relative higher Auslander correspondence to characterize the algebras satisfying the two-sided $(m+1,n+1)$-condition, see \cite{I2} for details. The following result is an easy case of \cite[Theorem 4.2.2]{I2}.

\begin{prop}Let $M$ be an $A$-module satisfying the conditions $(i)$, $(ii)$ and $(iii)$ in Theorem {\rm 3.3} and $\Lambda={\End}_{A}M$. If $m<n$ and $A$ has a minimal $M$-coresolution $0 \rightarrow A \rightarrow M_{0} \rightarrow M_{1} \rightarrow \cdots \rightarrow M_{m} \rightarrow 0$ with ${\pd}(\oplus^{m}_{i=0}M_{i})\leq m$, there exists an injective right $\Lambda$-module $J$ with projective dimension at most $m$ such that $J\text{-}{\dom}\,\Lambda_{\Lambda}\geq n+1$.
\end{prop}

\begin{proof}
Let $T$ be a basic $A$-module such that ${\add}\,T={\add}(\oplus^{m}_{i=0}M_{i})$. $T$ is an $m$-tilting $A$-module and $M\in T^{\perp}$. By the dual of of \cite[Theorem 4.2.2]{I2}, $\Lambda$ satisfies the two-sided $(m+1,n+1)$-condition. Thus there exists a minimal injective resolution $0 \rightarrow \Lambda_{\Lambda}\rightarrow J_{0}\rightarrow \cdots \rightarrow J_{n}\rightarrow J_{n+1}\rightarrow \cdots$ of right $\Lambda$-module $\Lambda$ with ${\pd}\,J_{i}\leq m$ for $0\leq i \leq n$. Take $J={\add}(\oplus^{n}_{i=0}J_{i})$, we have $J\text{-}{\dom}\,\Lambda_{\Lambda}\geq n+1$.
\end{proof}

\section{Almost n-minimal Auslander-Gorenstein algebras}

In \cite {AS}, Iyama and Solberg defined the $n$-minimal Auslander-Gorenstein algebra and characterized them as the endomorphism algebras of $n$-precluster tilting  modules. Recently, Adachi and Tsukamoto introduced the almost $n$-minimal Auslander-Gorenstein algebra in \cite{AT} which is a generalization of the $n$-minimal Auslander-Gorenstein algebra.  In this section, we define the almost $n$-precluster tilting module and characterize the almost $n$-minimal Auslander-Gorenstein algebra as the endomorphism algebra of the almost $n$-precluster tilting module.
Moreover, we describe the Gorenstein projective modules over the almost $n$-minimal Auslander-Gorenstein algebra in terms of the corresponding almost $n$-precluster tilting module.

\subsection{Almost n-precluster tilting modules}

In this subsection, we introduce and study the almost $n$-precluster tilting module.

\begin{definition}{\rm \cite[Definition 3.11]{AT}}Let $\Lambda$ be an Artin algebra and $I$ be the direct sum of all pairwise non-isomorphic indecomposable injective $\Lambda$-modules with projective dimension at most {\rm 1}. $\Lambda$ is called an almost $n$-minimal Auslander-Gorenstein algebra $($respectively, an almost $n$-Auslander algebra $)$ if it satisfies ${\id}\,\Lambda\leq n+1 \leq I\text{-}{\dom}\,\Lambda$ $(\text{respectively},\,{\gl}\,\Lambda\leq n+1 \leq I\text{-}{\dom}\,\Lambda\,)$ for an integer $n\geq0$.
\end{definition}

It is obvious that any $n$-minimal Auslander-Gorenstein algebra (${\id}\,\Lambda\leq n+1 \leq {\dom}\,\Lambda$)  is an almost $n$-minimal Auslander-Gorenstein algebra and any $n$-Auslander algebra ($\,{\gl}\,\Lambda\leq n+1 \leq {\dom}\,\Lambda\,$) is an almost $n$-Auslander algebra. Recall that an Artin algebra $\Lambda$ is called an $n$-Gorenstein (or Iwanaga-Gorenstein more precisely) algebra if ${\id}_{\Lambda}\Lambda={\id}\Lambda_{\Lambda}\leq n$. 1-Gorenstein algebra is an almost $m$-minimal Auslander-Gorenstein algebra for any integer $m\geq 0$. On the other hand, any almost $n$-minimal Auslander-Gorenstein algebra is $(n+1)$-Gorenstein according to \cite[Proposition 3.14 (1)]{AT}. In fact, an almost $n$-minimal Auslander-Gorenstein algebra $\Lambda$ is either a 1-Gorenstein algebra or an algebra with ${\id}\,\Lambda= n+1 = I\text{-}{\dom}\,\Lambda$ $(n\geq1)$.

Now we give the definition of the almost $n$-precluster tilting module.

\begin{definition}
Let $A$ be an Artin algebra and $M$ be an $A$-module. We call $M$ an almost $n$-precluster tilting module if it satisfies the following conditions

$(1)$ $M$ is a cogenerator over $A$.

$(2)$ ${\Ext}^{i}_{A}(M,M)=0$ for $1\leq i \leq n-1$.

$(3)$ $\tau_{n}M\in{\add}\,M$ where $\tau_{n}=\tau\Omega^{n-1}_{A}:A\text{-}\underline{\mo}\rightarrow A\text{-}\overline{\mo}$ is the $n$-Auslander-Reiten translation.

$(4)$ $M${\rm -codim}\,$A\leq 1$.
\end{definition}

Here we use the notion of almost $n$-precluster tilting module in correspondence to the notion of almost $n$-minimal Auslander-Gorenstein algebra introduced in \cite{AT}. Obviously any $n$-precluster tilting module is an almost $n$-precluster tilting module. Furthermore, we show the following relations between almost $n$-precluster tilting modules and $n$-precluster tilting modules.

\begin{prop}
Let $M$ be an almost $n$-precluster tilting module. Then $M$ is an $n$-precluster tilting module if and only if $A\in{\add}\,M$.
\end{prop}

\begin{proof}
If $M$ is an $n$-precluster tilting module, then $A\in{\add}\,M$ according to the definition of $n$-precluster tilting modules.

Conversely, assume $A\in{\add}\,M$. We only need to prove $\tau_{n}^{-}M\in{\add}\,M$. Since ${\Ext}^{i}_{A}(M,M)=0$ for $1\leq i \leq n-1$, $M\in \underline{^{\perp_{n-1}}A}\cap\overline{DA^{\perp_{n-1}}}$ holds. For an $A$-module $X$, denote by $|X|$ the number of pairwise non-isomorphic indecomposable direct summands of $X$. Note that $\tau_{n}:\underline{^{\perp_{n-1}}A}\rightarrow \overline{DA^{\perp_{n-1}}}$ and $\tau_{n}^{-}:\overline{DA^{\perp_{n-1}}}\rightarrow \underline{^{\perp_{n-1}}A}$ are equivalences. Then we have $|\tau_{n}M|+|DA|=|\tau_{n}(M/A)|+|A|=|M/A|+|A|=|M|$. Because $\tau_{n}M\oplus DA\in {\add}\,M$ and ${\add}\,\tau_{n}M\cap{\add}\,DA=0$, we know ${\add}\,M={\add}\,(\tau_{n}M\oplus DA)$. Thus $\tau_{n}^{-}M\in{\add}(\tau_{n}^{-}(\tau_{n}M\oplus DA))={\add}\,M/A\subseteq{\add}\,M$.
\end{proof}

\begin{remark}\rm
If $N$ is an $n$-precluster tilting $A$-module, we have $\tau_{n}^{-}N\in{\add}\,N$ and ${\Ext}^{i}_{A}(N,A)=0$ for $1\leq i \leq n-1$. For an almost $n$-precluster tilting $A$-module $M$,

$(1)$ $\tau_{n}^{-}M\in{\add}\,M$ does not necessarily hold. Let $A$ be the algebra given by the quiver
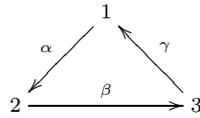
\begin{scriptsize}
\xymatrix{
                & 1 \ar[dl]_{\alpha} &            \\
 2  \ar[rr]^{\beta} & &     3\ar[ul]_{\gamma}       }
\end{scriptsize}
and the relations $\gamma\beta\alpha=\alpha\gamma\beta=0$.
Take $M=1\oplus{3\atop1}\oplus DA$. It is easily checked that $M$ is an almost 2-precluster tilting module. However, $\tau_{2}^{-}1={1 \atop 2}\notin {\add}\,M$.

$(2)$ If $n\geq3$, we have ${\Ext}^{i}_{A}(M,A)=0$ for $2\leq i \leq n-1$. This is obvious if $A\in {\add}\,M$. Otherwise, let $0 \rightarrow A \rightarrow M_{0} \rightarrow M_{1} \rightarrow 0$ be a minimal $M$-coresolution of $A$. Applying ${\Hom}_{A}(M,-)$ yields exact sequences ${\Ext}^{i-1}_{A}(M,M_{1})\rightarrow {\Ext}^{i}_{A}(M,A)\rightarrow {\Ext}^{i}_{A}(M,M_{0})$ for $2\leq i \leq n-1$. Thus we have ${\Ext}^{i}_{A}(M,A)=0$ for $2\leq i \leq n-1$.
\end{remark}

The following lemma is similar to a result which is contained in the proof of \cite[Proposition 3.9]{CS}. For the convenience of readers, we give a proof here.

\begin{lemma}
Let $M$ be an $A$-module and $\cdots \rightarrow P_{n} \rightarrow P_{n-1}\rightarrow \cdots \rightarrow P_{0} \rightarrow M \rightarrow 0$ be a minimal projective resolution of $M$. If ${\Ext}^{i}_{A}(M,M)=0$ for $1\leq i \leq n-1$, then $0\rightarrow {\Hom}_{A}(M,M) \rightarrow {\Hom}_{A}(P_{0},M)\rightarrow \cdots \rightarrow {\Hom}_{A}(P_{n-1},M)\rightarrow {\Hom}_{A}(P_{n},M)\rightarrow D{\Hom}_{A}(M,\tau_{n}M)\rightarrow 0$ is exact.
\end{lemma}

\begin{proof}
Applying ${\Hom}_{A}(-,M)$ to the minimal projective resolution of $M$, we get an exact sequence $0\rightarrow {\Hom}_{A}(M,M) \rightarrow {\Hom}_{A}(P_{0},M)\rightarrow \cdots \rightarrow {\Hom}_{A}(P_{n-1},M)\rightarrow {\Hom}_{A}(P_{n},M)$ since ${\Ext}^{i}_{A}(M,M)=0$ for $1\leq i \leq n-1$.

Consider the exact sequence $P_{n} \rightarrow P_{n-1}\rightarrow \Omega^{n-1}M \rightarrow 0$. According to \cite[Proposition 2.4 (a), \uppercase\expandafter{\romannumeral4}]{ASS}, there exists an exact sequence $0 \rightarrow \tau\Omega^{n-1}M=\tau_{n}M \rightarrow \upsilon P_{n} \rightarrow \upsilon P_{n-1}$. Applying $D{\Hom}_{A}(M,-)$, we get an exact sequence $D{\Hom}_{A}(M,\upsilon P_{n-1}) \rightarrow D{\Hom}_{A}(M,\upsilon P_{n})\rightarrow D{\Hom}_{A}(M,\tau_{n}M)\rightarrow 0$. Note that $D{\Hom}_{A}(M,\upsilon P_{i})\cong{\Hom}_{A}(P_{i},M)$ for $i=n-1,\,n$. Thus ${\Hom}_{A}(P_{n-1},M)\rightarrow {\Hom}_{A}(P_{n},M)\rightarrow D{\Hom}_{A}(M,\tau_{n}M)\rightarrow 0$ is exact and this completes our proof.
\end{proof}

For an almost $n$-precluster tilting module $M$, according to its definition, we have $\tau_{n}M\oplus DA\in{\add}\,M$. Moreover we give the following result.

\begin{prop}
Let $M$ be an almost $n$-precluster tilting $A$-module. Then we have ${\add}\,M={\add}(\tau_{n}M\oplus DA)$.
\end{prop}

\begin{proof}
Let $\Lambda={\End}_{A}M$ and $_{\Lambda}I={\Hom}_{A}(A,M)$. According to Theorem 3.3, $I$ is an injective $\Lambda$-module with ${\pd}_{\Lambda}I\leq 1$. It is known that the functor ${\Hom}_{A}(M,-)$ induces an equivalence between ${\add}\,M$ and the subcategory $\mathcal{P}(\Lambda_{\Lambda})$ of all projective right $\Lambda$-modules. Since $\tau_{n}M\oplus DA \in {\add}\,M$, we have $|\tau_{n}M\oplus DA|=|{\Hom}_{A}(M,\tau_{n}M\oplus DA )|$.
By Lemma 4.5, we get the following injective resolution of $\Lambda$
$$0\rightarrow \Lambda \rightarrow {\Hom}_{A}(P_{0},M)\rightarrow \cdots \rightarrow {\Hom}_{A}(P_{n},M)\rightarrow D{\Hom}_{A}(M,\tau_{n}M)\rightarrow 0$$ with ${\Hom}_{A}(P_{i},M)\in{\add}\,_{\Lambda}I$ for $0\leq i \leq n$.
This implies that the injective $\Lambda$-module $D{\Hom}_{A}(M,\tau_{n}M)$ has finite projective dimension. Let $N$ be a basic $\Lambda$-module such that ${\add}\,N={\add}(I\oplus D{\Hom}_{A}(M,\tau_{n}M))$. $N$ is an injective $\Lambda$-module with finite projective dimension. Moreover, $N$ is a tilting $\Lambda$-module and then $|N|=|\Lambda|=|M|$. On the other hand, since $I\cong D{\Hom}_{A}(M,DA)$, we have $|N|=|I\oplus D{\Hom}_{A}(M,\tau_{n}M)|=|D{\Hom}_{A}(M,DA\oplus \tau_{n}M)|=|DA\oplus \tau_{n}M|$. Then $|DA\oplus \tau_{n}M|=|M|$ and thus we get ${\add}\,M={\add}(\tau_{n}M\oplus DA)$.
\end{proof}

\begin{corollary}
Any almost $1$-precluster tilting $A$-module is a $1$-precluster tilting $A$-module.
\end{corollary}

\begin{proof}
Let $M$ be an almost $1$-precluster tilting $A$-module. According to Proposition 4.3, we only need to prove $A\in{\add}\,M$. First we have ${\add}\,M={\add}(\tau M\oplus DA)$ by Proposition 4.6. Then consider the additive sub-bifunctor $F^{M}$ of ${\Ext}^{1}_{A}(-,-)$. We have $\mathcal{I}(F^{M})={\add}\,M$ and $\mathcal{P}(F^{M})={\add}(\tau^{-}M\oplus A)={\add}(\tau^{-}\tau M\oplus A)={\add}(M\oplus A)$. According to \cite[Corollary 1.6(3)]{AS1}, $|M|=|M\oplus A|$ holds and thus $A\in {\add}\,M$.
\end{proof}

The above corollary tells us that almost 1-precluster tilting modules coincide with 1-precluster tilting modules.

\subsection{The correspondence between almost n-precluster tilting modules and almost n-minimal Auslander-Gorenstein algebras}

In this subsection, we characterize almost $n$-minimal Auslander-Gorenstein algebras as the endomorphism algebras of almost $n$-precluster tilting modules.

First we show that given an almost $n$-precluster tilting module, we can construct an almost $n$-minimal Auslander-Gorenstein algebra.

\begin{thm}
Let $A$ be an Artin algebra, $M$ be an almost $n$-precluster tilting $A$-module $(n\geq 1)$ and $\Lambda={\End}_{A}M$. Then we have the following.

$(1)$ $_{\Lambda}I={\Hom}_{A}(A,M)$ is an injective $\Lambda$-module with ${\pd}_{\Lambda}I\leq1$ and $A\cong{\End}_{\Lambda}I$.

$(2)$ $\Lambda$ is an almost $n$-minimal Auslander-Gorenstein algebra with ${\id}\,\Lambda= n+1 = I\text{-}{\dom}\,\Lambda$.
\end{thm}

\begin{proof}
$(1)$ follows from Theorem 3.3. Moreover, we have $I\text{-}{\dom}\,\Lambda\geq n+1$ and ${\add}\,I$ contains all injective $\Lambda$-modules with projective dimension at most 1 according to Proposition 3.1. By Lemma 4.5, there exists an exact sequence $0\rightarrow \Lambda \rightarrow {\Hom}_{A}(P_{0},M)\rightarrow \cdots \rightarrow {\Hom}_{A}(P_{n-1},M)\rightarrow {\Hom}_{A}(P_{n},M)\rightarrow D{\Hom}_{A}(M,\tau_{n}M)\rightarrow 0 $
with ${\Hom}_{A}(P_{i},M)\in {\add}\,I$ for  $0\leq i \leq n$. Since $\tau_{n}M\in {\add}\,M$, $D{\Hom}_{A}(M,\tau_{n}M)$ is an injective $\Lambda$-module. This implies that ${\id}\,\Lambda\leq n+1$. Thus $\Lambda$ is an almost $n$-minimal Auslander-Gorenstein algebra.
\end{proof}

Conversely, an almost $n$-minimal Auslander-Gorenstein algebra can give rise to an almost $n$-precluster tilting module.

\begin{thm}
Let $\Lambda$ be an almost $n$-minimal Auslander-Gorenstein algebra with ${\id}\,\Lambda= n+1 = I\text{-}{\dom}\,\Lambda$ $(n\geq2)$ and $A={\End}_{\Lambda}I$. Then we have the following.

$(1)$ $_{A}M={\Hom}_{\Lambda}(\Lambda,I)$ is an almost $n$-precluster tilting $A$-module.

$(2)$ $\Lambda\cong{\End}_{A}M$.
\end{thm}

\begin{proof}
By Theorem 3.6, we only need to prove $\tau_{n}M\in{\add}\,M$. Let $0\rightarrow \Lambda \xrightarrow{f_{0}}I_{0}\xrightarrow{f_{1}} \cdots \xrightarrow{f_{n}}I_{n}\xrightarrow{f_{n+1}}I_{n+1}\rightarrow 0$ with $I_{j}\in{\add}\,I$ for $0\leq j \leq n$ be a minimal injective resolution of $\Lambda$ and $U=\text{ker}f_{n}\oplus I$. Then $U$ is a cotilting $\Lambda$-module with ${\id}\,U=2$ according to \cite[Theorem 3.17]{AT} and $I$ is a dualizing summand of $U$. Consider the additive sub-bifunctor $F=F_{{\Hom}_{\Lambda}(U,I)}$ of ${\Ext^{1}_{A}(-,-)}$. By Theorem 2.3, $M={\Hom}_{\Lambda}(\Lambda,I)$ is an $F$-cotilting $A$-module with ${\id}_{F}M=0$. Then we have $\mathcal{I}(F)={\add}\,M$. On the other hand, $\mathcal{I}(F)={\add}(\tau{\Hom}_{\Lambda}(U,I)\oplus DA)$ holds. Thus we get $\tau{\Hom}_{\Lambda}(U,I)\in{\add}\,M$.

Consider the exact sequence $0\rightarrow \Lambda \rightarrow I_{0}\cdots \rightarrow I_{n-2}\rightarrow \text{ker}f_{n}\rightarrow 0$. Since $I$ is injective, applying ${\Hom}_{\Lambda}(-,I)$ yields the following exact sequence
$$0\rightarrow {\Hom}_{\Lambda}(\text{ker}f_{n},I)\rightarrow {\Hom}_{\Lambda}(I_{n-2},I)\rightarrow \cdots {\Hom}_{\Lambda}(I_{0},I)\rightarrow M\rightarrow 0$$ where ${\Hom}_{\Lambda}(I_{j},I)$ are projective $A$-modules for $0\leq j \leq n-2$. Then we have $\tau{\Hom}_{\Lambda}(U,I)=\tau{\Hom}_{\Lambda}(\text{ker}f_{n},I)=\tau\Omega^{n-1}M=\tau_{n}M$ and this completes our proof.
\end{proof}

\begin{remark}\rm
In the above theorem we assume $n\geq2$ to make sure the existence of a 2-cotilting $\Lambda$-module $U$ with a dualizing summand $I$. When $n=1$, $\Lambda$ is an almost $2$-minimal Auslander-Gorenstein algebra. By Theorem 3.6 and Lemma 4.5, $M$ satisfies (i) $DA\in {\add}\,M$ and $M\text{-codim}\,A\leq1$, (ii) $\Lambda\cong{\End}_{A}M$ and (iii) ${\Hom}_{A}(M,\tau M)$ is a projective right $\Lambda$-module. Conversely given an $A$-module $M$ satisfying (i),(ii) and (iii), it is easy to check that $\Lambda={\End}_{A}M$ is an almost 2-minimal Auslander-Gorenstein algebra and $A\cong{\End}_{\Lambda}I$.
\end{remark}

Combining Theorem 4.8 and Theorem 4.9, we give the main result in this section.

\begin{thm}Assume $n\geq2$. There exists a bijection between the equivalence classes of almost n-precluster tilting modules over Artin algebras and the Morita-equivalence classes of almost n-minimal Auslander-Gorenstein algebras.
\end{thm}

Since the relative dominant dimension of an algebra is not left-right symmetric, the almost $n$-minimal Auslander-Gorenstein algebra is also not left-right symmetric. For an almost $n$-minimal Auslander-Gorenstein algebra $\Lambda$, $\Lambda^{\op}$ is not always an almost $n$-minimal Auslander-Gorenstein algebra. Or equivalently, $J_{\Lambda}\text{-}{\dom}\,\Lambda_{\Lambda}\geq n+1$ does not always hold where ${\add}\,J$ consists of all injective right $\Lambda$-modules with projective dimension at most 1. We give a sufficient condition for $\Lambda$ being left-right symmetric.

\begin{prop}
Assume $n\geq2$. Let $M$ be an almost $n$-precluster tilting module over an Artin algebra $A$ and $0\rightarrow A \rightarrow M_{0}\rightarrow M_{1}\rightarrow 0$ be a minimal $M$-coresolution of $A$. If ${\pd}_{A}(M_{0}\oplus M_{1})\leq 1$, $\Lambda^{\op}=({\End}_{A}M)^{\op}$ is an almost $n$-minimal Auslander-Gorenstein algebra.
\end{prop}

\begin{proof}
By Proposition 3.8, we have $J_{\Lambda}\text{-}{\dom}\,\Lambda_{\Lambda}\geq n+1$. Since $\Lambda$ is $(n+1)$-Gorenstein, ${\id}\,\Lambda_{\Lambda}\leq n+1$ holds. Thus $\Lambda^{\op}$ is an almost $n$-minimal Auslander-Gorenstein algebra.
\end{proof}

\subsection{Gorenstein projective modules over almost n-minimal Auslander-Gorenstein algebras}

Let $\Lambda$ be an almost $n$-minimal Auslander-Gorenstein algebra with ${\id}\,\Lambda= n+1 = I\text{-}{\dom}\,\Lambda$ $(n\geq 1)$ and $A={\End}_{\Lambda}I$. We describe the subcategory $\mathcal{GP}(\Lambda)$ of Gorenstein projective $\Lambda$-modules in terms of the corresponding $A$-module $_{A}M={\Hom}_{\Lambda}(\Lambda,I)$.

\begin{thm}Let $F=F^{M}$ be an additive sub-bifunctor of ${\Ext}^{1}_{A}(-,-)$. Then we have the following.

$(1)$ $M$ is an $F$-tilting $A$-module and $M^{\perp_{F}}$ is a Frobenius category with projective-injective objects ${\add}\,M$.

$(2)$ If $n\geq 2$, $M$ is an almost $n$-precluster tilting $A$-module with ${\pd}_{F}M\leq n-1$ and $M^{\perp_{F}}=M^{\perp_{n-1}}$.

$(3)$ ${\Hom}_{A}(-,M)$ induces a duality between $M^{\perp_{F}}$ and $\mathcal{GP}(\Lambda)$. Moreover, $M^{\perp_{F}}/[M]$ and $\underline{\mathcal{GP}(\Lambda)}^{\op}$ are equivalent as triangulated categories.
\end{thm}

\begin{proof}
(1) By Theorem 3.6, $M$ is a cogenerator over $A$ and $\Lambda\cong{\End}_{A}M$. Then we have $\mathcal{I}(F)={\add}(M\oplus DA)={\add}\,M$ and $M$ is an $F$-cotilting $A$-module. Since $\Lambda$ is Gorenstein, according to Proposition 2.4, $M$ is an $F$-tilting $A$-module and $A$ is $F$-Gorenstein. By the dual of \cite[Theorem 3.24]{AS2} and the dual of \cite[Proposition 2.3]{AS2}, there exists an $F$-exact sequence $0\rightarrow Y \rightarrow M_{X}\rightarrow X \rightarrow 0$ with $M_{X}\in{\add}\,M$ for any $X\in M^{\perp_{F}}$. Note that $M$ is a projective object in $M^{\perp_{F}}$. Then $M^{\perp_{F}}$ has enough projective-injective objects ${\add}\,M$ with respect to $F$-exact sequences. Thus $M^{\perp_{F}}$ is a Frobenius category where short exact sequences are precisely $F$-exact sequences. Let $[M]$ be the ideal of morphisms factoring through an object in ${\add}\,M$. It follows that $M^{\perp_{F}}/[M]$ is a triangulated category.

(2) If $n\geq2$, $M$ is an almost $n$-precluster tilting $A$-module by Theorem 4.11. Then we have ${\add}\,M={\add}(\tau_{n}M\oplus DA)$ and $\mathcal{P}(F)={\add}(\tau^{-}M\oplus A)={\add}({\tau^{-}\tau_{n}M}\oplus A)={\add}(\Omega^{n-1}M\oplus A)$. Since ${\Ext}^{i}_{A}(M,M)=0$ for $1\leq i \leq n-1$, the first $n$ terms $0 \rightarrow \Omega^{n-1}M \rightarrow P_{n-2}\rightarrow \cdots \rightarrow P_{1}\rightarrow P_{0} \rightarrow M\rightarrow 0$ of a minimal projective resolution of $M$ is an $F$-projective resolution of $M$. Then ${\pd}_{F}M\leq n-1$ and this implies that $M^{\perp_{F}}=\{X\in A\text{-}{\rm mod}|\,{\Ext}_{F}^{i}(M,X)=0 \, {\rm for}\, 1\leq i \leq n-1\}$. On the other hand, ${\Ext}^{i}_{F}(M,X)={\Ext}^{i}_{A}(M,X)$ holds for $1\leq i \leq n-1$ and any $A$-module $X$ according to \cite[Proposition 2.5(1)]{IS}. Thus we get $M^{\perp_{F}}=M^{\perp_{n-1}}$.

(3) Since $\mathcal{I}(F)={\add}(M\oplus DA)={\add}\,M$, we have $^{\perp_{F}}M=A\text{-}{\mo}$. By Theorem 2.2, ${\Hom}_{A}(-,M)$ induces a duality between $A$-mod and $^{\perp}U$ where $U={\Hom}_{A}(\mathcal{P}(F),M)$ is a cotilting $\Lambda$-module. For a Gorenstein $\Lambda$-module $N$, ${\Ext}^{i}_{\Lambda}(N,U)=0$ holds for $i\geq 1$ since $U$ has finite injective dimension. Thus there exists an $A$-module $L$ such that $N\cong{\Hom}_{A}(L,M)$. According to Theorem 2.2(3), we have ${\Ext}^{i}_{\Lambda}(N,\Lambda)={\Ext}^{i}_{\Lambda}({\Hom}_{A}(L,M),{\Hom}_{A}(M,M))\cong{\Ext}^{i}_{F}(M,L)$ for $i\geq 1$. Note that $\Lambda$ is Gorenstein and $\mathcal{GP}(\Lambda)=^{\perp}$$\Lambda$. Then $N$ is Gorenstein projective if and only if $L\in M^{\perp_{F}}$. Since the functor ${\Hom}_{A}(-,M)$ is a fully faithful, it induces an duality between $M^{\perp_{F}}$ and $\mathcal{GP}(\Lambda)$. Moreover, we get a triangle equivalence between $M^{\perp_{F}}/[M]$ and $\underline{\mathcal{GP}(\Lambda)}^{\op}$.
\end{proof}

It is known that a Gorenstein algebra has finite global dimension if and only if every Gorenstein projective module is projective. Immediately we get the following characterization of almost $n$-Auslander algebras.

\begin{corollary}
Keep the notations in Theorem {\rm 4.13}. Then the following are equivalent.

$(1)$ $\Lambda$ is an almost $n$-Auslander algebra.

$(2)$ $\Lambda$ has finite global dimension.

$(3)$ $M^{\perp_{F}}={\add}\,M$
\end{corollary}

\begin{proof}
$(1)\Longleftrightarrow(2)$: This follows from \cite[Proposition 3.15(1)]{AT}.

$(2)\Longleftrightarrow(3)$: We have a duality ${\Hom}_{A}(-,M):{\add}\,M\rightarrow \mathcal{P}(\Lambda)$. Note that $\Lambda$ is a Gorenstein algebra. By Theorem 4.13(3), $M^{\perp_{F}}={\add}\,M$ if and only if $\mathcal{GP}(\Lambda)=\mathcal{P}(\Lambda)$ if and only if ${\gl}\,\Lambda<\infty$.
\end{proof}

For an $n$-precluster tilting $A$-module $N$, $^{\perp_{n-1}}N=N^{\perp_{n-1}}$ holds according to \cite[Proposition 3.8]{IS}. We give a similar result for an almost $n$-precluster tilting $A$-module $M$ by calculating the relative dominant dimension of Gorenstein projective $\Lambda$-modules.

\begin{prop}
Keep the notations in Theorem {\rm 4.13}. We have the following.

$(1)$ $I\text{-}{\dom}\,X\geq n+1$ holds for any Gorenstein projective $\Lambda$-module $X$.

$(2)$ If $n\geq2$, $M^{\perp_{n-1}}\subseteq $$^{\perp_{n-1}}M$ holds.
\end{prop}

\begin{proof}
(1) Since $\Lambda$ is an almost $n$-minimal Auslander-Gorenstein algebra, we have $I\text{-}{\dom}\,P\geq n+1$ for any projective $\Lambda$-module $P$. Let $Y_{1}$ be in $\Omega({\Lambda\text{-}{\mo}})$. There exists an exact sequence $0\rightarrow Y_{1}\rightarrow P_{1}\rightarrow Y^{'}_{1} \rightarrow 0$ where $P_{1}$ is a projective $\Lambda$-module. By \cite[Corollary 3.2(1)]{CX}, $I\text{-}{\dom}\,Y_{1}\geq\text{min}\{I\text{-}{\dom}\,P_{1}, \,I\text{-}{\dom}\,Y_{1}^{'}+1\}\geq 1$. Let $Y_{2}$ be in $\Omega^{2}(\Lambda\text{-}{\mo})$. Then there exists an exact sequence $0\rightarrow Y_{2} \rightarrow P_{2} \rightarrow Y_{2}^{'}\rightarrow 0$ where $P_{2}$ is a projective $\Lambda$-module and $Y^{'}_{2}$ is in $\Omega({\Lambda\text{-}{\mo}})$.  By \cite[Corollary 3.2(1)]{CX} again, we have $I\text{-}{\dom}\,Y_{2}\geq\text{min}\{I\text{-}{\dom}\,P_{2}, \,I\text{-}{\dom}\,Y_{2}^{'}+1\}\geq 2$. Continuing this procedure, we get     $I\text{-}{\dom}\,Y_{n+1}\geq n+1$ for any $Y_{n+1}$ in $\Omega^{n+1}(\Lambda\text{-}{\mo})$.

Note that $\mathcal{GP}(\Lambda)=\Omega^{n+1}(\Lambda\text{-}{\mo})$ since $\Lambda$ is $(n+1)$-Gorenstein. Then $I\text{-}{\dom}\,X\geq n+1$ holds for any $X$ in $\mathcal{GP}(\Lambda)$.

(2) Let $X$ be in $M^{\perp_{n-1}}$. According to Theorem 4.13, ${\Hom}_{A}(X,M)$ is a Gorenstein projective $\Lambda$-module. Then $I\text{-}{\dom}\,{\Hom}_{A}(X,M)\geq n+1$ by (1). On the other hand, we have $I\text{-}{\dom}\,{\Hom}_{A}(X,M)={\rm inf}\{i\geq1\mid {\Ext}^{i}_{A}(X,M)\neq0\}+1$ by Proposition 3.5. Thus ${\rm inf}\{i\geq1\mid {\Ext}^{i}_{A}(X,M)\neq0\}\geq n$ holds. This implies $X\in$$ ^{\perp_{n-1}}M$.
\end{proof}

\subsection{Almost n-cluster tilting modules}
In Corollary 4.14 we get a class of almost $n$-precluster tilting modules $M$ with $M^{\perp_{F}}={\add}\,M$ which correspond to almost $n$-Auslander algebras. Since almost $n$-Auslander algebras are a generalization of $n$-Auslander algebras and $n$-cluster tilting modules correspond to $n$-Auslander algebras (see \cite{I2}), we get the following generalization of $n$-cluster tilting modules.

\begin{definition}Let $A$ be an Artin algebra and $M$ be an $A$-module. We call $M$ an almost $n$-cluster tilting $A$-module if it satisfies ${\add}\,M=M^{\perp_{n-1}}$, $\tau_{n}M\in{\add}\,M$ and $M\text{-}{\rm codim}\,A\leq 1$.
\end{definition}

The almost 1-cluster tilting module is just $A$-mod for a representation-finite algebra $A$. For $n\geq2$, by Theorem 4.11 and Corollary 4.14, there is a one-to-one correspondence between the equivalence classes of almost $n$-cluster tilting modules over Artin algebras and the Morita-equivalence classes of almost $n$-Auslander algebras.

Obviously we have the relation for the following three classes of modules:

\noindent \{$n$-cluster tilting modules\} $\subseteq$ \{almost $n$-cluster tilting modules\} $\subseteq$ \{almost $n$-precluster tilting modules\}.

Fix an integer $n$, a natural question is whether the numbers of pairwise non-isomorphic indecomposable direct summands of modules in each class are equal. Iyama proved in \cite{I2} that all 2-cluster tilting modules over an Artin algebra share the same number of pairwise non-isomorphic indecomposable direct summands. In \cite{CS}, Chen and K\"{o}nig generalized Iyama's result to maximal 2-precluster tilting modules. We give the following result for almost 2-cluster tilting modules.

\begin{prop}
Let $M$ and $N$ be almost $2$-cluster tilting modules over an Artin algebra $A$. Then $\Lambda_{M}={\End}_{A}M$ and $\Lambda_{N}={\End}_{A}N$ are derived equivalent. In particular, $|M|=|N|$ holds.
\end{prop}

To prove Proposition 4.17, we need the following lemma.

\begin{lemma}
Let $M$ be an almost $n$-cluster tilting module over an Artin algebra $A$ and $F=F^{M}$ be an additive sub-bifunctor of ${\Ext}^{1}_{A}(-,-)$. Then we have ${\gl}_{F}A\leq n-1$.
\end{lemma}

\begin{proof}
If $n=1$, $A$ is representation-finite and ${\add}\,M=A$-mod. Then $\mathcal{I}(F)=A$-mod and ${\gl}_{F}A=0$.

Now assume $n\geq 2$. Obviously $M$ is also an almost $n$-precluster tilting $A$-module. According to Theorem 4.13 (1) and (2), $M$ is an $F$-tilting $A$-module with ${\pd}_{F}M\leq n-1$ and $M^{\perp_{F}}=M^{\perp_{n-1}}$. By the dual of \cite[Theorem 3.2(a)]{AS2}, $M^{\perp_{F}}\text{-}{\rm codim}\,X\leq n-1$ holds for any $A$-module $X$. On the other hand, we have ${\add}\,M=M^{\perp_{n-1}}$ since $M$ is an almost $n$-cluster tilting $A$-module. Note that $\mathcal{I}(F)={\add}\,M$, then ${\id}_{F}X \leq n-1$ holds for any $A$-module $X$. This implies ${\gl}_{F}A\leq n-1$.
\end{proof}

Now we are ready to prove Proposition 4.17.

\begin{proof}[Proof of Proposition 4.17] Since $M$ is an almost 2-cluster tilting $A$-module, by Lemma 4.18, we have ${\gl}_{F^{M}}A\leq 1$. Then ${\id}_{F^{M}}N\leq 1$ and there exists an $F^{M}$-exact sequence $0\rightarrow N \rightarrow M_{0} \rightarrow M_{1} \rightarrow 0$ with $M_{0},\,M_{1}\in {\add}\,M$. Applying ${\Hom}_{A}(-,M)$ yields a projective resolution $0 \rightarrow {\Hom}_{A}(M_{1},M)\rightarrow {\Hom}_{A}(M_{0},M)\rightarrow {\Hom}_{A}(N,M)\rightarrow 0 $ of $\Lambda_{M}$-module ${\Hom}_{A}(N,M)$. This shows that ${\pd}\,{\Hom}_{A}(N,M)\leq 1$.

Consider the following commutative diagram
\begin{footnotesize}
\[ \xymatrix{
0\ar[r] &((N,M),(N,M))\ar[d]^\cong \ar[r] &((M_{0},M),(N,M)) \ar[d]^\cong \ar[r] &((M_{1},M),(N,M))\ar[d]^\cong \ar[r]& 0\\
0\ar[r] & (N,N)  \ar[r] & (N,M_{0}) \ar[r] & (N,M_{1})\ar[r] & 0 .& }
\]
\end{footnotesize}The vertical maps are isomorphisms since the functor ${\Hom}_{A}(-,M)$ is fully faithful. The lower row is exact due to ${\Ext}^{1}_{A}(N,N)=0$. Then the upper row is also exact. This implies that ${\Ext}^{1}_{\Lambda_{M}}({\Hom}_{A}(N,M),{\Hom}_{A}(N,M))=0$. So ${\Hom}_{A}(N,M)$ is a partial tilting $\Lambda_{M}$-module and we have $|N|=|{\Hom}_{A}(N,M)|\leq |\Lambda_{M}|=|M|$. Similarly, we can get $|M|\leq|N|$. Thus $|M|=|N|$ holds and ${\Hom}_{A}(N,M)$ is a 1-tilting $\Lambda_{M}$-module. Note that  ${\End}_{\Lambda_{M}}{\Hom}_{A}(N,M)\cong{\End}_{A}N={\Lambda_{N}}$, then $\Lambda_{M}$ and $\Lambda_{N}$ are derived equivalent.
\end{proof}

\begin{remark}\rm
${\Hom}_{A}(-,M)$ sends almost 2-cluster tilting $A$-modules to 1-tilting $\Lambda_{M}$-modules. Moreover, it is easy to check that ${\Hom}_{A}(-,M)$ induces a duality between $A$-mod and the subcategory of $\Lambda_{M}$-modules with projective dimension at most 1. This is a generalization of
\cite[Proposition 4.4]{GLS}.
\end{remark}

\section*{Acknowledgements}
The authors would like to thank the referees for valuable suggestions and comments. The first author is grateful to Professor Nan Gao for helpful discussions. This work is supported by the National Natural Science Foundation of China (11671230).

\newpage
Shen Li: School of Mathematics, Shandong University, PR China

\emph{E-mail address}: fbljs603@163.com

Shunhua Zhang: School of Mathematics, Shandong University, PR China

\emph{E-mail address}: shzhang@sdu.edu.cn


\begin{thebibliography}{100}

\bibitem{APR}\, M. Auslander, M.I. Platzeck, I. Reiten, Coxter functors without diagrams, Trans. Amer. Math. Soc 250(1979),1-46.

\bibitem{ARS}\, M. Auslander, I. Reiten, S.O. Smal${\o}$,  Representation Theory of Artin Algebras, Cambridge Studies in Advanced Math, 36. Cambridge University Press, Cambridge, 1995.

\bibitem{AS1}\, M. Auslander, ${\O}$. Solberg, Relative homology and representation theory I:  Relative homology and homologically finite subcategories, Comm. Algebra 21 (9) (1993) 2995-3031.

\bibitem{AS2}\, M. Auslander, ${\O}$. Solberg, Relative homology and representation theory II: Relative cotilting theory, Comm. Algebra 21 (9)
(1993) 3033-3079.

\bibitem{AS3}\, M. Auslander, ${\O}$. Solberg, Relative homology and representation theory III: Cotilting modules and Wedderburn correspondence,
Comm. Algebra 21 (9) (1993) 3081-3097.

\bibitem{AS}\, M. Auslander, ${\O}$. Solberg, Gorenstein algebras and algebras with dominant dimension at least 2, Comm. Algebra 21(11), 3897-3934 (1993).

\bibitem{ASS}\, I. Assem, D. Simson, A. Skowronski, Elements of the representation theory of associative algebras. Vol.1. Techniques of representation theory.  London Mathematical Society Student Texts, 65. Cambridge University Press, Cambridge, 2006.

\bibitem{AT}\, T. Adachi, M. Tsukamoto, Tilting modules and dominant dimension with respect to injective modules, arXiv:1902.09185v2.

\bibitem{CS}\, H. Chen, S. K\"{o}nig, Orhto-symmetric modules, Gorenstein algebras and derived equivalence, Int. Math. Res. Not. IMRN(22)(2016) 6979-7037.

\bibitem{CX}\, H. Chen, C. Xi, Dominant dimension, derived equivalences and tilting modules, Israel J. Math. 215(2016), no.1, 349-395.

\bibitem{GLS}\, C. Geiss, B. Leclerc, J. Schr\"{o}er, Rigid modules over preprojective algebras, Invent. Math 165(2006), 589-632.

\bibitem{I1}\, O. Iyama, $\tau$-Categories III: Auslander orders and Auslander-Reiten quivers, Algebr. Represent. Theory 8 (2005), no. 5, 601-619.

\bibitem{I2}\, O. Iyama, Auslander correspondence, Adv. Math. 210 (2007), no. 1, 51-82.

\bibitem{I3}\, O. Iyama, Cluster tilting for higher Auslander algebras, Adv. Math. 226(2011),1-61.

\bibitem{IS}\, O. Iyama, ${\O}$. Solberg, Auslander-Gorenstein algebras and precluster tilting, Adv. Math. 326 (2018), 200-240.

\bibitem{BJM}\, B.J. M\"{u}ller, The classification of algebras by dominant dimension, Canad. J. Math. 20(1968), 398-409.

\bibitem{M}\,  R. Marczinzik, On stable modules that are not Gorenstein projective, arXiv:1709.01132v3.

\bibitem{JM}\, J. Miyachi, Injective resolutions of Noetherian rings and cogenerator, Proc. Amer. Math. Soc. 128(8)(2000) 2233-2242.

\bibitem{NRTZ}\, V.C. Nguyen, I. Reiten, G. Todorov, S. Zhu, Dominant dimension and tilting modules, to appear in Math. Z, arXiv:1706.00475.

\bibitem{PS}\, M. Pressland, J. Sauter, Special tilting modules for algebras with positive dominant dimension, arXiv:1705.03367.


\end{thebibliography}
\end{document}